\documentclass[12pt, reqno]{amsart}
\usepackage[dvipsnames]{xcolor}
\usepackage{mathtools}
\usepackage{etoolbox}
\usepackage{amsthm, graphicx, microtype}
\usepackage[export]{adjustbox}
\usepackage{blindtext}
\usepackage{latexsym}
\usepackage[pagewise]{lineno}
\usepackage[utf8]{inputenc}
\usepackage{exscale}
\usepackage{amsmath}
\usepackage{amssymb}
\usepackage{amsfonts}
\usepackage{mathrsfs}
\usepackage{amsbsy}
\usepackage{relsize}
\usepackage{comment}
\PassOptionsToPackage{reqno}{amsmath}
\usepackage{esint} 
\usepackage{amssymb} 
\usepackage{stmaryrd}
\usepackage{cite} 

\usepackage{tikz}

\usepackage{amsthm, graphicx, microtype}
\usepackage[export]{adjustbox}

\usepackage{overpic}
\usepackage[T2A,T1]{fontenc}

\newcommand{\cd}{\mathscr{D}}
\newcommand{\bi}{\mathbf{I}}
\newcommand{\pdt}{\partial_t} 

\definecolor{mypink1}{rgb}{0.858, 0.188, 0.478}
\definecolor{mypink2}{RGB}{219, 48, 122}
\definecolor{mypink3}{cmyk}{0, 0.7808, 0.4429, 0.1412}
\definecolor{mygray}{gray}{0.6}
\definecolor{venetianred}{rgb}{0.78, 0.03, 0.08}
\definecolor{sapphire}{rgb}{0.03, 0.15, 0.4}
\definecolor{utahcrimson}{rgb}{0.83, 0.0, 0.25}
\definecolor{trueblue}{rgb}{0.0, 0.45, 0.81}
\definecolor{carminered}{rgb}{1.0, 0.0, 0.22}
\definecolor{cobalt}{rgb}{0.0, 0.28, 0.67}
\definecolor{cornflowerblue}{rgb}{0.39, 0.58, 0.93}
\definecolor{darkmagenta}{rgb}{0.55, 0.0, 0.55}
\definecolor{electricultramarine}{rgb}{0.25, 0.0, 1.0}
\definecolor{falured}{rgb}{0.5, 0.09, 0.09}
\definecolor{hancornflowerblue}{rgb}{0.32, 0.09, 0.98}
\definecolor{mahogany}{rgb}{0.75, 0.25, 0.0}
\definecolor{oucrimsonred}{rgb}{0.6, 0.0, 0.0}
\definecolor{persianblue}{rgb}{0.11, 0.22, 0.73}
\definecolor{rufous}{rgb}{0.66, 0.11, 0.03}
\definecolor{uablue}{rgb}{0.0, 0.2, 0.67}
\definecolor{zaffre}{rgb}{0.0, 0.08, 0.66}
\definecolor{carmine}{rgb}{0.59, 0.0, 0.09}

\definecolor{BrickRed}{rgb}{0.58, 0.0, 0.83}

\definecolor{armygreen}{rgb}{0.29, 0.33, 0.13}
\definecolor{brass}{rgb}{0.71, 0.65, 0.26}
\definecolor{antiquefuchsia}{rgb}{0.57, 0.36, 0.51}
\definecolor{amethyst}{rgb}{0.6, 0.4, 0.8}
\definecolor{mauvetaupe}{rgb}{0.57, 0.37, 0.43}
\usepackage[colorlinks, linkcolor=carminered, citecolor=electricultramarine, urlcolor=mahogany, hypertexnames=false]{hyperref}

\newcommand{\bS}{\mathbf{S}}

\newcommand{\bk}{\mathbf{k}}
\newcommand{\bM}{\mathbf{M}}
\newcommand{\R}{\mathbb{R}}

\newcommand{\ts}{\mathsf{s}}

\newtheorem{defi}{Definition}[section]

\newtheorem{thm}{Theorem}[section]
\newtheorem{rem}{Remark}[section]
\newtheorem{lem}{Lemma}[section]
\newtheorem{cor}{Corollary}[section]

\renewcommand{\theequation}{\thesection.\arabic{equation}}

\numberwithin{equation}{section}

\usepackage{lipsum} 
\usepackage{tikz,xcolor}

\definecolor{lime}{HTML}{A6CE39}
\DeclareRobustCommand{\orcidicon}{
\begin{tikzpicture}
\draw[lime, fill=lime] (0,0) 
circle [radius=0.16] 
node[white] {{\fontfamily{qag}\selectfont \tiny ID}};
\draw[white, fill=white] (-0.0625,0.095) 
circle [radius=0.007];
\end{tikzpicture}
\hspace{-2mm}
}
\foreach \x in {A, B}{\expandafter\xdef\csname orcid\x\endcsname{\noexpand\href{https://orcid.org/\csname orcidauthor\x\endcsname}
{\noexpand\orcidicon}}
}

\usepackage{bm}

\title[Subordinated Fractional Diffusion Equations]{Long-Time Asymptotics for Subordinated Fractional Diffusion Equations}

\author[M. Majdoub, E. Mliki]{Mohamed Majdoub\orcidA{} \& Ezzedine Mliki\orcidB{}}
\address{Department of Mathematics, College of Science, Imam Abdulrahman Bin Faisal University, P. O. Box 1982, Dammam, Saudi Arabia}
\address{Basic and Applied Scientific Research Center, Imam Abdulrahman Bin Faisal University, P.O. Box 1982, 31441, Dammam, Saudi Arabia}
\email{\sl \color{blue}{mohamed.majdoub@fst.rnu.tn}}
\email{\sl \color{blue}{med.majdoub@gmail.com}}
\email{\sl \color{blue}{mmajdoub@iau.edu.sa}}
\email{\sl \color{blue}{ermliki@iau.edu.sa}}

\begin{document}

\begin{abstract} 
We study the long-time behavior of solutions to a class of evolution equations arising from random-time changes driven by subordinators. Our focus is on fractional diffusion equations involving mixed local and nonlocal operators. By combining techniques from probability theory, asymptotic analysis, and partial differential equations (PDEs), we characterize the dynamics of the subordinated solutions. This approach extends classical fractional dynamics and establishes a deeper connection between stochastic processes and deterministic PDEs.
\end{abstract}




\maketitle

\renewcommand{\theequation}{\thesection.\arabic{equation}}

\section{ Introduction}
\label{Intro}
In recent years, fractional diffusion equations have emerged as essential in capturing the dynamics of complex systems characterized by memory and nonlocality. In this work, we investigate the long-time behavior of solutions to a class of linear evolution equations influenced by random time changes, modeled using a variety of subordinators, such as \emph{gamma} and \emph{$\alpha$-stable} subordinators. These subordinators, denoted by $\bS = \left\{\bS_t\right\}_{t \geq 0}$, introduce temporal randomness that significantly shapes the dynamics of the underlying systems over long time periods.

Our main focus lies in understanding the behavior of the \emph{subordinated solution} $v^E(x, t)$, which arises from applying subordination to certain evolution partial differential equations (PDEs). Specifically, we analyze large-time dynamics by convolving the original solution with the density function $G_t$ associated with the subordinator $\bS$. To extract meaningful asymptotic information, we study the \emph{Cesàro mean} of the subordinated solution, which provides a robust measure of its average behavior. 

Understanding these subordinated dynamics is not only of mathematical interest, but is also highly relevant for modeling real-world systems. In particular, Our results offer insight into how biological time may unfold in natural processes such as species evolution and ecological development. This perspective enables the development of more refined models in fields such as population dynamics, epidemiology, and ecosystem analysis. For more context and foundational background, we refer the reader to \cite{MS2012, MK2000}.

Specifically, we consider the fractional diffusion equation
\begin{equation}\label{LEE}
\left \{ \begin{array}{rl}
\partial_t^\alpha v(x, t)&=t^{\gamma}\mathscr{L}_{a, b}\, v(x,t)\quad  \mbox{ in }\quad  \R^N\times (0,\infty),\\
v(x,0)&=\varphi(x)\geq 0 \quad\mbox{ in }\quad \mathbb{R}^N,
\end{array}\right.
\end{equation}

where $\mathscr{L}_{a, b}= a\Delta - b(-\Delta)^s$ is a mixed local-nonlocal operator, $a, b, \gamma \geq 0$, $0<\alpha\leq 1$, $0<\ts<1$, and $\varphi(x)$ is a nonnegative initial datum. Here, $\partial_t^\alpha$ denotes the Caputo fractional derivative~\cite{Diethelm} while $(-\Delta)^s$ represents the fractional Laplacian, defined via the Fourier transform as
$$
(-\Delta)^s u := \mathcal{F}^{-1}\left(|\xi|^{2s}\mathcal{F}(u)\right),
$$
where $\mathcal{F}$ and $\mathcal{F}^{-1}$ denote the Fourier transform and its inverse, respectively (see, e.g.,~\cite{M}).

Let $ \bS $ be a subordinator with an associated density function $ G_t $. Given a solution $ v $ to \eqref{LEE} such that $ v(x,\cdot) \in L^1(0, \infty) $, we define its \emph{subordination} with respect to $ G_t $ by
\begin{equation}\label{v-E}
v^{E}(x,t) := \int_0^\infty v(x,\tau) G_t(\tau)\,d\tau, \quad (x,t) \in \mathbb{R}^N \times (0,\infty),
\end{equation}
where the integral is absolutely convergent due to the assumed time integrability of $ v $.

In the particular case when $\alpha = 1$ and $\gamma = 0$, it is known---under suitable assumptions (see \cite{KKS20, KKS21})---that the subordinated solution $v^{E}(x,t)$ satisfies the following fractional evolution equation:
\begin{equation}\label{FDE}
\begin{cases}
\left(\mathbb{D}^{(k)}_t v^{E}\right)(x,t) = \mathscr{L}_{a,b}\,v^{E}(x,t), & (x,t) \in \mathbb{R}^N \times (0,\infty),\\[4pt]
v^{E}(x,0) = \varphi(x), & x \in \mathbb{R}^N.
\end{cases}
\end{equation}
Here, $\mathbb{D}^{(k)}_t$ denotes a convolution-type derivative defined by
\begin{equation}\label{D-t-k}
\left(\mathbb{D}^{(k)}_t g\right)(t)
= \frac{d}{dt}\int_0^t \mathbf{k}(t-s)\,\big(g(s)-g(0)\big)\,ds,
\end{equation}
where the kernel $\mathbf{k}\in L^1_{\mathrm{loc}}(0,\infty)$ is assumed to be positive and locally integrable.

The primary objective of this paper is to identify and characterize particular classes of subordinators that facilitate the analysis of time-asymptotic behavior in generalized fractional dynamics. Specifically, we investigate the long-time behavior of the subordinated solution $v^{E}(x, t)$.

To carry out this analysis, we consider the \emph{Cesàro mean} of the subordinated solution $v^{E}(x, t)$, defined by
\begin{equation}
\label{Cesaro}
\bM_t(v^{E}(x,t)) := \frac{1}{t}\int_0^t v^{E}(x,s)\,ds.
\end{equation}
Using the subordination formula~\eqref{v-E}, we observe that
\begin{equation}
\label{Cesaro-1}
\bM_t(v^{E}(x,t)) = \int_0^\infty v(x,\tau) \bM_t(G_t(\tau))\, d\tau,
\end{equation}
where $G_t$ denotes the density function associated with the subordinator $\bS$.

Our  approach is grounded in the use of Laplace transform techniques, the Feller--Karamata Tauberian theorem, and the time-integrability properties of the underlying solution $v(x,t)$. This framework applies to a variety of models, including the time-space-fractional diffusion equation \eqref{LEE}. For additional background and related work, see \cite{MMAS-2025, KKS20, KKS21, Lu2018}.

We consider a broad family of admissible kernels $\bk \in L^1_{\text{loc}}(0,\infty)$ satisfying the following conditions on their Laplace transforms $\mathcal{K}(\lambda)$:
\begin{equation}
\label{Adm-1}
\lambda \mathcal{K}(\lambda) \xrightarrow{\lambda \to 0^+} 0,
\end{equation}
and, for some $\varrho \geq 0$,
\begin{equation}
\label{Adm-2}
L(x) := x^{-\varrho} \mathcal{K}(x^{-1}) \quad \text{is slowly varying},
\end{equation}
in the sense that
$$
\lim_{x \to \infty} \frac{L(\lambda x)}{L(x)} = 1, \quad \text{for all } \lambda > 0.
$$

A well-known and widely used example of such kernels is given by
\begin{equation}
\label{k-mu}
\bk(s) = \int_0^1 \frac{s^{-\sigma}}{\Gamma(1-\sigma)} \mu(\sigma)\, d\sigma,
\end{equation}
where $\mu: [0,1] \to (0, \infty)$ is a continuous function. For further properties and asymptotic analysis of kernels of the form \eqref{k-mu}, we refer the reader to \cite{Koch}. In particular, the Laplace transform $\mathcal{K}(\lambda)$ satisfies
\begin{equation}
\label{Lap-K}
\mathcal{K}(\lambda) = \int_0^1 \lambda^{\sigma - 1} \mu(\sigma)\, d\sigma,
\end{equation}
with the corresponding slowly varying function
\begin{equation}
\label{L-SVF}
L(x) = \int_0^1 x^{-\sigma} \mu(\sigma)\, d\sigma.
\end{equation}

We now list several fundamental examples of kernel classes that have been used throughout our analysis.
\begin{enumerate}
    \item \textsf{Stable Subordinator Class (${\bf C}_1$):}
    $$
    \mathcal{K}(\lambda) = \lambda^{\theta - 1}, \quad 0 < \theta < 1.
    $$
    
    \item \textsf{Distributed Order Derivative Class (${\bf C}_2$):}
    $$
    \mathcal{K}(\lambda) \sim C \lambda^{-1}(-\ln \lambda)^{-\kappa}, \quad C, \kappa > 0, \quad \text{as } \lambda \to 0^+.
    $$
    
    \item \textsf{Inverse Gamma Subordinator Class (${\bf C}_3$):}
    $$
    \mathcal{K}(\lambda) = \frac{\sqrt{b}}{\lambda} \left(2\sqrt{2\lambda + a} - \sqrt{a} \right), \quad a \geq 0,\, b > 0.
    $$
    
    \item \textsf{Gamma Subordinator Class (${\bf C}_4$):}
    $$
    \mathcal{K}(\lambda) = \frac{a}{\lambda} \ln\left(1 + \frac{\lambda}{b}\right), \quad a, b > 0.
    $$
    
    \item \textsf{Tempered Stable Subordinator Class (${\bf C}_5$):}
    $$
    \mathcal{K}(\lambda) = \frac{(\lambda + \beta)^\theta - \beta^\theta}{\lambda}, \quad \beta > 0, \; 0 < \theta < 1.
    $$
\end{enumerate}

\begin{rem}
\label{Classes}
\hfill{\rm 
\begin{itemize}
    \item[(i)] All kernel classes $\mathbf{C}_i$, $1 \leq i \leq 5$, satisfy the condition \eqref{Adm-1}, except for $\mathbf{C}_3$ with $a > 0$. However, for $a = 0$, $\mathbf{C}_3$ reduces to the stable case $\mathbf{C}_1$ with $\theta = 1/2$.
    
    \item[(ii)] Condition \eqref{Adm-2} holds for all classes $\mathbf{C}_i$ with the following values of $\varrho$ and corresponding slowly varying functions:
    \begin{itemize}
        \item[$\bullet$] $\mathbf{C}_1$: $\varrho = 1 - \theta$, $L(x) = 1$,
        \item[$\bullet$] $\mathbf{C}_2$: $\varrho = 1$, $L(x) = C (\ln x)^{-\kappa}$,
        \item[$\bullet$] $\mathbf{C}_3$: $\varrho = 1$, $L(x) = \sqrt{b}\left(2\sqrt{2/x + a} - \sqrt{a}\right)$,
        \item[$\bullet$] $\mathbf{C}_4$: $\varrho = 0$, $L(x) = a x \ln\left(1 + \frac{1}{bx}\right)$,
        \item[$\bullet$] $\mathbf{C}_5$: $\varrho = 0$, $L(x) = x\left(\left(1/x + \beta\right)^\theta - \beta^\theta\right)$.
    \end{itemize}
\end{itemize}}
\end{rem}

Henceforth, we denote by $\|\cdot\|_p$ the standard Lebesgue norm in the space $L^p$, applied to both spatial and temporal variables, for all $1 \leq p \leq \infty$.

We now present our first main result.
\begin{thm}
\label{Main-thm}
Suppose that for every nonnegative initial datum $0 \leq \varphi \in L^1(\mathbb{R}^N) \cap L^\infty(\mathbb{R}^N)$, the solution $v$ of \eqref{LEE} is nonnegative and satisfies $v(x, \cdot) \in L^1(0, \infty)$. If, in addition, conditions \eqref{Adm-1} and \eqref{Adm-2} hold, then the Cesàro mean $\bM_t(v^{E}(x,t))$ of the subordinate solution $v^{E}(x,t)$ (defined in \eqref{v-E}) has the following asymptotic behavior as $t \to \infty$:
\begin{equation}
    \label{main-res}
    \Gamma(\varrho+1)\,\bM_t(v^{E}(x,t))\, \underset{t\to \infty }{\sim}\, \|v(x,\cdot)\|_{1} \,\bigg(t^{-1} \mathcal{K}\left(t^{-1}\right)\bigg).
\end{equation}
\end{thm}
\begin{rem}
~{\rm 
\begin{itemize}
\item[(i)] The case $\alpha = 1$, $\gamma = 0$, $a = 1$, and $b = 0$ was studied in \cite{KKS20, KKS21} for arbitrary dimensions.
\item[(ii)] The special case where $\alpha = 1$, $\gamma = 0$, $a = 0$, $b = 1$, with $\theta \in (0, 2)$, was recently studied in \cite{MMAS-2025}.
    \item[(iii)] The asymptotic result \eqref{main-res} can be reformulated as follows:
    \begin{equation}
    \label{main-res-equivalent}
    \bM_t(v^{E}(x,t))\, \underset{t\to \infty }{\sim}\, \frac{\|v(x,\cdot)\|_{1}}{\Gamma(\varrho+1)} \,\left(\int_0^1 t^{-\sigma} \mu(\sigma)d\sigma\right),
\end{equation}
where the function $\mu$ is given by \eqref{k-mu}.
    \item[(iv)] Theorem~\ref{Main-thm} provides a unified framework for analyzing a broad class of operators and kernels.
\end{itemize}
}  
\end{rem}
For particular kernel classes, Theorem~\ref{Main-thm} yields explicit asymptotic formulas as stated below.
\begin{cor}
    \label{Spec-Kern}
    Under the hypotheses of Theorem~\ref{Main-thm} (with the exception of class $\mathbf{C}_3$, where the $L^1$ condition on $v(x,\cdot)$ is not required), the long-time behavior of the Cesàro mean takes the following precise forms as $t \to \infty$:
    \begin{equation}
    \label{Class-Examp}
    \begin{split}
\bM_t(v^{E}(x,t))\, \underset{t\to \infty }{\sim}\, \left\{\begin{array}{llll}   \frac{\|v(x,\cdot)\|_1}{\Gamma(2-\theta)}\,t^{-\theta} \quad &\text{if}\quad \mathbf{k}\in\; (\mathbf{C}_1),\\
\mathtt{C}\,\|v(x,\cdot)\|_1\,\left(\ln t\right)^{-\kappa} \quad &\text{if}\quad \mathbf{k}\in\; (\mathbf{C}_2),\\
 2\sqrt{ab}\,\|v(x,\cdot)\|_{1,\sqrt{ab}}  \quad &\text{if}\quad \mathbf{k}\in\; (\mathbf{C}_3),\\
 a\|v(x,\cdot)\|_1\, (bt)^{-1} \quad &\text{if}\quad \mathbf{k}\in\; (\mathbf{C}_4),\\
 \theta\beta^{\theta-1}\,\|v(x,\cdot)\|_1\, t^{-1} \quad &\text{if}\quad \mathbf{k}\in\; (\mathbf{C}_5).
\end{array}\right.
\end{split}
\end{equation}
Here, we use the notation
$$
\|v(x,\cdot)\|_{1,\ell}:=\int_0^\infty\,{\rm e}^{-\ell\tau}\,v(x,\tau)\,d\tau.
$$
\end{cor}

For the fractional diffusion equation with power time coefficient
\begin{equation}\label{alpha-1-0}
\partial_t^\alpha v(x,t) = t^\gamma\Delta v(x,t), \quad v(x,0) = \varphi(x),
\end{equation}
where $\alpha \in (0,1)$, $\gamma \geq 0$, we obtain the following asymptotic behavior.
\begin{thm}
    \label{alpha1-0}
Let $N \geq 1$, $\gamma \geq 0$, and $\frac{2}{N(\gamma+1)} < \alpha < 1$. Consider the operator $\mathscr{L}_{1,0}$ and a nonnegative initial datum $\varphi \in L^1(\mathbb{R}^N) \cap L^\infty(\mathbb{R}^N)$. Then the solution $v$ to \eqref{alpha-1-0} remains nonnegative and is integrable in time, that is, $v \geq 0$ and $v(x,\cdot) \in L^1(0,\infty)$ for every $x \in \mathbb{R}^N$. Moreover, under the additional assumptions on the kernel given by \eqref{Adm-1} and \eqref{Adm-2}, the solution exhibits the asymptotic behavior described by \eqref{main-res}.
\end{thm}
\begin{rem}
    {\rm The time-dependent coefficient $t^\gamma$ allows our analysis to cover all dimensions $N \geq 1$ and any fractional order $\alpha \in (0,1)$. Specifically, for given $N \geq 1$ and $\alpha \in (0,1)$, the condition $\gamma > \frac{2}{N\alpha} - 1$ guarantees the validity of our results throughout this full parameter range.}
\end{rem}
For the mixed local-nonlocal diffusion equation with time-dependent coefficient
\begin{equation}\label{1-1-1}
\partial_t v(x,t) = t^\gamma(\Delta - (-\Delta)^\ts)v(x,t), \quad v(x,0) = \varphi(x),
\end{equation}
where $\gamma \geq 0$ and $\ts \in (0,1)$, we establish the following asymptotic result.

\begin{thm}\label{11-1}
Let $N \geq 1$, $\gamma \geq 0$, and $0 < \ts < \frac{N(\gamma+1)}{2}$. For any nonnegative initial datum $\varphi \in L^1(\mathbb{R}^N) \cap L^\infty(\mathbb{R}^N)$, the solution $v$ to \eqref{1-1-1} remains nonnegative and is integrable in time. Moreover, under the kernel assumptions \eqref{Adm-1} and \eqref{Adm-2}, the solution exhibits the precise asymptotic behavior
$$
\bM_t(v^E(x,t)) \sim \frac{\|v(x,\cdot)\|_1}{\Gamma(\varrho+1)} t^{\varrho-1}L(t) \quad \text{as } t \to \infty,
$$
as characterized in \eqref{main-res}.
\end{thm}
\begin{rem}
{\rm
~\begin{itemize}
    \item[$(i)$] For dimensions $N \geq 2$, the condition $\ts < \frac{N(\gamma+1)}{2}$ holds automatically for all $\ts\in (0,1)$.
    \item[$(ii)$] In the case $N = 1$, the inequality $\ts < \frac{\gamma+1}{2}$ is satisfied provided $\gamma \geq 1$.
\end{itemize}}
\end{rem}

The paper is structured as follows. In Section~\ref{Prelim}, we present the necessary preliminaries and introduce key notation used throughout the work. Section~\ref{PMR} contains the proofs of the main results, including Theorem~\ref{Main-thm} and Theorems~\ref{alpha1-0}--\ref{11-1}, which offer a comprehensive asymptotic analysis for the cases $\alpha \in (0,1)$ with parameters $(a,b) = (1,0)$ and $\alpha = 1$ with $(a,b) = (1,1)$. Lastly, Section~\ref{CR} concludes with final remarks and a discussion of open problems.

\section{Preliminaries}
\label{Prelim}
Throughout the remainder of this article, we adopt the following notational conventions for clarity and conciseness.
\begin{itemize}
    \item We write $X \lesssim Y$ (or equivalently $Y \gtrsim X$) to indicate that there exists a constant $C > 0$ such that $X \leq C Y$.
    \item The notation $f(y) \underset{y \to y_0}{\sim} g(y)$ signifies that $\displaystyle\lim_{y \to y_0} \,\frac{f(y)}{g(y)} = 1$, where $y_0 \in [-\infty, \infty]$.
    \item The Laplace transform of a function $f: [0,\infty) \to \mathbb{R}$ is defined by
    $$
    \mathcal{L}(f)(\lambda) := \int_0^\infty e^{-\lambda t} f(t)\,dt, \quad \lambda > 0.
    $$
    \item We denote the $L^1$-norm over the interval $(0,\infty)$ by $\|\cdot\|_1 := \|\cdot\|_{L^1(0,\infty)}$.
    \item Throughout this work, the letter $C$ stands for a generic positive constant that may vary between different expressions.
\end{itemize}

\subsection{Subordinators and Their Properties}

\begin{defi}
A \emph{subordinator} $\bS = \{\bS_t\}_{t\geq 0}$ is a non-decreasing Lévy process taking values in $[0,\infty)$, characterized by stationary independent increments and almost sure right-continuous paths with $\bS_0 = 0$. \end{defi}
These processes provide fundamental models for random time changes and jump phenomena \cite{Bern, Ber, Bertoin}.

The infinite divisibility of $\bS_t$ implies that its Laplace transform takes the form:
\begin{equation}\label{laplace-sub}
\mathbb{E}[e^{-\lambda \bS_t}] = e^{-t\Phi(\lambda)} = e^{-t\lambda\mathcal{K}(\lambda)}, \quad \lambda \geq 0,
\end{equation}
where $\Phi(\lambda)$ is the \emph{Laplace exponent}, a Bernstein function \cite{Bernstein,Wissem}. The Lévy-Khintchine representation gives:
\begin{equation}\label{levy-khintchine}
\Phi(\lambda) = \int_{0}^\infty (1 - e^{-\lambda\tau}) d\sigma(\tau),
\end{equation}
with Lévy measure $\sigma$ satisfying $\int_0^\infty (1 \wedge \tau) d\sigma(\tau) < \infty$.

The associated kernel $\mathbf{k}$ and Lévy measure $\sigma$ are related through:
\begin{equation}
\mathbf{k}(t) = \sigma((t,\infty)), \quad t \geq 0,
\end{equation}
yielding the Laplace transform relation:
\begin{equation}
\mathcal{K}(\lambda) = \frac{\Phi(\lambda)}{\lambda} = \int_0^\infty e^{-\lambda t}\mathbf{k}(t) dt.
\end{equation}

The \emph{inverse subordinator} $E_t := \inf\{s \geq 0 : \bS_s \geq t\}$ has marginal density $G_t(\tau)$ satisfying:
\begin{equation}
G_t(\tau)d\tau = \partial_\tau \mathbb{P}(E_t \leq \tau) = -\partial_\tau \mathbb{P}(\bS_\tau < t).
\end{equation}

\subsubsection{Important Subordinator Classes}
\begin{itemize}
\item[i)] {\sf Tempered Stable Subordinator.}
For $\theta \in (0,1)$, $\beta > 0$:
\begin{equation}
f_\beta(x,t) = e^{-\beta x + \beta^\theta t}f(x,t),
\end{equation}
where $f(x,t)$ is the $\theta$-stable density. The Laplace transform is:
\begin{equation}
\mathcal{L}[f_\beta(\cdot,t)](\lambda) = e^{-t((\lambda+\beta)^\theta - \beta^\theta)} = e^{-t\lambda\mathcal{K}(\lambda)}.
\end{equation}
\item[ii)] {\sf $\theta$-Stable Subordinator.}
For $\theta \in (0,1)$, the kernel and its Laplace transform are:
\begin{equation}
\mathbf{k}(t) = \frac{t^{-\theta}}{\Gamma(1-\theta)}, \quad \mathcal{K}(\lambda) = \lambda^{\theta-1}.
\end{equation}
This satisfies condition (\ref{Adm-1}) with $\varrho = 1-\theta$.
\item[iii)] {\sf Gamma Subordinator.}
With parameters $a,b>0$:
\begin{equation}
\mathbf{k}(t) = a\Gamma(0,bt), \quad \mathcal{K}(\lambda) = \frac{a}{\lambda}\ln\left(1 + \frac{\lambda}{b}\right),
\end{equation}
where $\Gamma(\nu,x) := \int_x^\infty t^{\nu-1}e^{-t}dt$ is the incomplete Gamma function. This satisfies (\ref{Adm-1}) with $\varrho=0$.

\item[iv)] {\sf Inverse Gamma Subordinator.}
For $a\geq 0$, $b>0$:
\begin{equation}
\mathbf{k}(t) = \sqrt{\frac{b}{2\pi}}\left(\frac{2}{\sqrt{t}}e^{-at/2} - \sqrt{2a\pi}(1 - \operatorname{erf}(\sqrt{at/2}))\right),
\end{equation}
with Laplace transform:
\begin{equation}
\mathcal{K}(\lambda) = \frac{\sqrt{b}}{\lambda}\left(2\sqrt{2\lambda + a} - \sqrt{a}\right).
\end{equation}
When $a>0$, (\ref{Adm-1}) holds with limit $\sqrt{ab}$ instead of 0.
\end{itemize}
We end this section with a key relationship between the Laplace transform of the kernel $\bk$ and the density $G_t$.
\begin{lem}\label{G-K}
The Laplace transform of $G_t(\tau)$ satisfies:
\begin{equation}\label{G-K-lambda}
\int_0^\infty e^{-\lambda t}G_t(\tau) dt = \mathcal{K}(\lambda)e^{-\tau\lambda\mathcal{K}(\lambda)},
\end{equation}
where $\mathcal{K}(\lambda) = \mathcal{L}(\mathbf{k})(\lambda)$.
\end{lem}
\begin{proof}
See \cite[Lemma 2.2]{KKS20}.
\end{proof}
\subsection{Caputo Fractional Derivative}

We provide a concise introduction to the Caputo fractional derivative. For a comprehensive treatment of fractional calculus, we refer the reader to \cite{Diethelm, KST, IP, SKM}.

The \emph{Riemann-Liouville fractional integral} of order $\alpha \in (0,1)$ is defined as:
\begin{equation}  
\label{RL-Int}  
\bi^{\alpha}u(t) = \frac{1}{\Gamma(\alpha)} \int_0^t (t-s)^{\alpha-1} u(s) \, ds,  
\end{equation}  
where $\Gamma$ denotes the gamma function, given by
\begin{equation*}
\Gamma(z) = \int_0^\infty s^{z-1} e^{-s} \, ds, \quad \Re(z) > 0.
\end{equation*}

The \emph{Riemann-Liouville fractional derivative} of order $\alpha \in (0,1)$ is then defined as:
\begin{equation}  
\label{RL-FD}  
\cd^{\alpha}u(t) = \frac{d}{dt} \left( \bi^{1-\alpha}u(t) \right) = \frac{1}{\Gamma(1-\alpha)} \frac{d}{dt} \left( \int_0^t (t-s)^{-\alpha} u(s) \, ds \right),  
\end{equation}  
assuming that the function $t \mapsto \bi^{1-\alpha}u(t)$ is absolutely continuous.

The \emph{Caputo fractional derivative} of order $\alpha \in (0,1)$ is defined as:
\begin{equation}  
\label{Cap-FD1}  
\pdt^{\alpha}u(t) = \cd^{\alpha} \left( u(t) - u(0) \right),  
\end{equation}  
where $u(0)$ exists and $t \mapsto \bi^{1-\alpha}u(t)$ is absolutely continuous.

Under the stronger assumption that $u$ is absolutely continuous, the \emph{Caputo} derivative admits the equivalent representation \cite[Theorem 2.1, p. 92]{KST}:
\begin{equation}  
\label{Cap-FD2}  
\pdt^{\alpha}u(t) = \frac{1}{\Gamma(1-\alpha)} \int_0^t (t-s)^{-\alpha} \frac{du}{ds}(s) \, ds.  
\end{equation}

\subsection{Mixed Heat Semigroup}
The fractional Laplacian operator $(-\Delta)^{\ts}$ with $\ts>0$ generates a semigroup $\{e^{-t(-\Delta)^{\ts}}\}_{t\geq 0}$ whose kernel $\mathscr{E}_{\ts}$ is smooth, radial, and satisfies the scaling property:
\begin{equation}\label{kernel}
\mathscr{E}_{\ts}(x,t) = t^{-\frac{N}{2\ts}} \mathscr{K}_\ts(t^{-\frac{1}{2\ts}}x),
\end{equation}
where the profile function $\mathscr{K}_\ts$ is given by the Fourier integral
\begin{equation}
    \label{K-s}
    \mathscr{K}_\ts(x) = (2\pi)^{-N/2}\int_{\R^N} e^{i x\cdot\xi} e^{-|\xi|^{2\ts}} d\xi.
\end{equation}

Explicit expressions for $\mathscr{E}_{\ts}$ are known in two important cases:
\begin{itemize}
    \item[(i)] For $\ts=1$ (standard heat kernel):
    \begin{equation}
        \label{Gauss}
        \mathscr{E}_{1}(x,t) = (4\pi t)^{-N/2} e^{-\frac{|x|^2}{4t}}, \quad \mathscr{K}_{1}(x) = (4\pi)^{-N/2} e^{-\frac{|x|^2}{4}}.
    \end{equation}
    
    \item[(ii)] For $\ts=1/2$ (Poisson kernel):
    \begin{equation}
        \label{Poisson}
        \mathscr{E}_{1/2}(x,t) = \frac{\Gamma(\frac{N+1}{2}) t}{\pi^{\frac{N+1}{2}} (t^2 + |x|^2)^{\frac{N+1}{2}}}, \quad \mathscr{K}_{1/2}(x) = \frac{\Gamma(\frac{N+1}{2})}{\pi^{\frac{N+1}{2}} (1 + |x|^2)^{\frac{N+1}{2}}}.
    \end{equation}
\end{itemize}

For general $\ts \in (0,1)$, while no explicit expression exists, we have the following crucial positivity estimate.
\begin{lem}\label{Positive}
For $N \geq 1$ and $\ts \in (0,1)$, the profile function satisfies
\begin{equation}\label{kthetaest}
(1+|x|)^{-N-2\ts} \lesssim \mathscr{K}_{\ts}(x) \lesssim  (1+|x|)^{-N-2\ts},\quad x \in \R^N.
\end{equation}
Consequently, $\mathscr{K}_{\ts} \in L^p(\R^N)$ for all $1 \leq p \leq \infty$.
\end{lem}

The proof appears in \cite[p.~395]{Alonso2021}, with the positivity result first stated without proof in \cite[p.~263]{BG1960}. A detailed treatment can also be found in \cite[Theorem 2.1]{BG1960}.

The operator $\mathscr{L} = \mathscr{L}_{1,1}$ generates a strongly continuous contraction semigroup $\{e^{t\mathscr{L}}\}_{t\geq 0}$ on $L^2(\mathbb{R}^N)$, with each operator $e^{t\mathscr{L}}$ given by convolution with the fundamental solution $\mathbf{E}_{\ts}(t)$. This fundamental solution $\mathbf{E}_{\ts}(x,t)$, which solves the evolution equation 
\begin{equation}
    \label{eq-1-1-1}
    \partial_t u(x,t) = \mathscr{L}u(x,t),\quad (x,t)\in\R^N\times (0,\infty),
\end{equation}
with Dirac mass as initial data, can be expressed as the convolution of two kernels: the classical heat kernel $\mathscr{E}_{1}(x,t) = (4\pi t)^{-N/2} e^{-|x|^2/(4t)}$ and the fractional heat kernel $\mathscr{E}_{\ts}(x,t)$ defined in \eqref{kernel}. 

The fundamental solution $\mathbf{E}_{\ts}(x,t)$ possesses several important properties that we now summarize (see, e.g., \cite{Kirane2025}).
\begin{lem}\label{E-s}
For any $\ts \in (0,1)$, the following hold:
\begin{enumerate}
    \item \textsf{Regularity and positivity}: The solution $\mathbf{E}_{\ts}$ belongs to $C^\infty(\mathbb{R}^N \times (0,\infty))$ and satisfies $\mathbf{E}_{\ts}(x,t) \geq 0$ for all $(x,t) \in \mathbb{R}^N \times (0,\infty)$.
    \item \textsf{Mass conservation}: The kernel preserves total mass, with
   \begin{equation}
       \label{Mass-conv-E}
       \int_{\mathbb{R}^N} \mathbf{E}_{\ts}(x,t)\,dx = 1,\quad t > 0.
   \end{equation}
    \item \textsf{Smoothing estimates}: For any  $\varphi \in L^r(\mathbb{R}^N)$ and $1 \leq r \leq q \leq \infty$, we have
    \begin{equation}\label{eq:conv_est}
    \|\mathbf{E}_{\ts}(t) \ast \varphi\|_{q} \leq C \min\left\{t^{-\frac{N}{2}\left(\frac{1}{r}-\frac{1}{q}\right)}, t^{-\frac{N}{2\ts}\left(\frac{1}{r}-\frac{1}{q}\right)}\right\} \|\varphi\|_{r},\quad t>0.
    \end{equation}
\end{enumerate}
\end{lem}

We now consider the fundamental solution of the equation
\begin{equation}  
    \label{alpha11}  
    \partial_t^{\alpha} u(x,t) = \Delta u(x,t),\quad (x,t)\in\R^N\times (0,\infty).  
\end{equation}  
It follows directly that the solution of \eqref{alpha11}, with initial data $ u(x,0) = \varphi(x) $, can be expressed as  
\begin{equation}  
    \label{representation}  
    u(x,t) = \int_{\mathbb{R}^N} \mathbf{Z}_\alpha(y,t) \varphi(x-y) \, dy = \left[\mathbf{Z}_\alpha(t) \ast \varphi\right](x),  
\end{equation}  
where $\mathbf{Z}_\alpha(x,t)$ is the fundamental solution of \eqref{alpha11}.  

Applying the Fourier transform with respect to the spatial variable $ x \in \mathbb{R}^N $ to \eqref{alpha11} and the initial condition $ u(x,0) = \delta(x) $ (the Dirac mass at the origin) yields the following initial value problem for a linear ODE:  
\begin{equation}  
    \label{ODE}  
    \partial_t^{\alpha} \widehat{\mathbf{Z}_{\alpha}}(\xi,t) + |\xi|^2 \widehat{\mathbf{Z}_{\alpha}}(\xi,t) = 0, \quad \widehat{\mathbf{Z}_{\alpha}}(\xi,0) = 1.  
\end{equation}  
As established in the literature (see, e.g., \cite{Diethelm, KST, Lu1999}), the unique solution of \eqref{ODE} is given by  
\begin{equation}  
    \label{sol-ode}  
    \widehat{\mathbf{Z}_{\alpha}}(\xi,t) = \mathbb{E}_\alpha\left(-t^{\alpha}|\xi|^2\right),  
\end{equation}  
where $\mathbb{E}_\alpha$ denotes the Mittag-Leffler function, defined as  
\begin{equation}  
    \label{MitLef}  
    \mathbb{E}_\alpha(z) = \sum_{n=0}^\infty \frac{z^n}{\Gamma(1 + n \alpha)}, \quad z \in \mathbb{C}.  
\end{equation}  

Taking the inverse Fourier transform, we obtain  
\begin{equation}  
    \label{Z-alpha}  
    \mathbf{Z}_{\alpha}(x,t) = \frac{1}{(2\pi)^N} \int_{\mathbb{R}^N} e^{i x \cdot \xi} \mathbb{E}_\alpha\left(-t^{\alpha}|\xi|^2\right) d\xi, \quad (x,t)\in\R^N\times (0,\infty).  
\end{equation}  

The fundamental solution $\mathbf{Z}_{\alpha}(x,t)$ satisfies several key properties, summarized in the following lemma:  
\begin{lem}\label{Z-alpha-properties}  
For any $\alpha \in (0,1)$, the following hold:  
\begin{enumerate}  
    \item The solution $\mathbf{Z}_{\alpha}$ belongs to $C^\infty(\mathbb{R}^N \times (0,\infty))$.  
    \item The kernel preserves the total mass, with  
 \begin{equation}
     \label{Mass-conv-Z}
     \int_{\mathbb{R}^N} \mathbf{Z}_{\alpha}(x,t) \, dx = 1, \quad  t > 0. 
 \end{equation} 
    \item There exists a constant $C > 0$ such that  
    \begin{equation}\label{decay-Z-alpha}  
    0 \leq \mathbf{Z}_{\alpha}(x,t) \leq C t^{-\frac{N\alpha}{2}}, \quad (x,t)\in\R^N\times (0,\infty).  
    \end{equation}  
\end{enumerate}  
\end{lem}
For a detailed proof of Lemma \ref{Z-alpha-properties}, we refer the reader to \cite{Kemp1}.
\subsection{Karamata's Tauberian theorem}
The Feller-Karamata Tauberian theorem represents a fundamental result in asymptotic analysis, providing a deep connection between the asymptotic behavior of a function and its Laplace transform. For comprehensive treatments of the general result, we refer to \cite[Section~1.7]{Bingham} and \cite[Chapter~XIII, Section~5]{Feller}. Below, we provide a version of this theorem that is particularly well-suited to our applications.
\begin{thm}
\label{FKT}
Let $\mathbb{F}: (0,\infty) \to \mathbb{R}$ be a right-continuous, non-decreasing function whose Laplace transform $\mathcal{L}(\mathbf{F})(\lambda)$ exists for all $\lambda > 0$. Then the following statements are equivalent:

\begin{equation}
\label{U-t}
\mathbf{F}\left(\frac{1}{t}\right) \sim \frac{C t^{-\varrho}}{\Gamma(\varrho+1)} L\left(\frac{1}{t}\right) \quad \text{as } t \to 0^+,
\end{equation}

and

\begin{equation}
\label{w-lambda}
\mathcal{L}(\mathbf{F})(\lambda) \sim C \lambda^{-1-\varrho} L\left(\frac{1}{\lambda}\right) \quad \text{as } \lambda \to 0^+.
\end{equation}

Here, $C > 0$ is a positive constant, $\varrho \geq 0$ is the index of regular variation, and $L$ is a slowly varying function at infinity.
\end{thm}
\begin{rem}
~
\begin{itemize}
\item[(i)] {\rm Theorem~\ref{FKT} provides a precise correspondence between the asymptotic behavior of $\mathbf{F}\left(\frac{1}{t}\right)$ as $t \to 0^+$ and its Laplace transform $\mathcal{L}(\mathbf{F})(\lambda)$ as $\lambda \to 0^+$, quantified through the parameters $C > 0$, $\varrho \geq 0$, and the slowly varying function $L$.}
    \item[(ii)] {\rm A typical example is obtained when $\mathbf{F}(\tau) = \tau^{\varrho}/\Gamma(\rho+1)$ for $\varrho \geq 0$, which yields $L(t) \equiv 1$ (trivially slowly varying) and $\mathcal{L}(\mathbf{F})(\lambda) = \lambda^{-1-\varrho}$, verifying the theorem's conclusion in this special case.}
    \end{itemize}
\end{rem}

\section{Proof of the main results}
\label{PMR}
\subsection{Proof of Theorem \ref{Main-thm}} The proof proceeds along the same lines as in \cite[Theorem 1.1]{MMAS-2025}. For the sake of completeness and clarity, we present the full argument below.

We define
\begin{equation}  
\label{eq:F_def}  
\mathbf{F}(t) = \int_0^t v^{E}(x,s)\,ds = t \mathbf{M}_t(v^{E}(x,s)),  
\end{equation}  
where $\mathbf{M}_t$ denotes the Ces\`{a}ro mean (see \eqref{Cesaro}).

Applying formula \eqref{G-K-lambda}, we obtain
\begin{equation}  
\label{eq:wK_relation}  
\frac{\lambda\,w(\lambda)}{\mathcal{K}(\lambda)} = \int_0^\infty e^{-\tau \lambda \mathcal{K}(\lambda)} v(x, \tau)\,d\tau,  
\end{equation}  
where $w(\lambda) := \mathcal{L}(\mathbf{F})(\lambda)$ is the Laplace transform of $\mathbf{F}$.

Since $v(x,\cdot) \in L^1(0,\infty)$ and by \eqref{Adm-1}, the dominated convergence theorem yields
\begin{equation}  
\label{eq:lambda_limit}  
\lim_{\lambda \to 0^+} \left( \frac{\lambda w(\lambda)}{\mathcal{K}(\lambda)} \right) = \int_0^\infty v(x, \tau) \, d\tau = \|v(x,\cdot)\|_{1}.  
\end{equation}  
Consequently, we have the asymptotic relation
\begin{equation}  
\label{eq:w_asymptotic}  
w(\lambda) \underset{\lambda \to 0^+}{\sim} \|v(x,\cdot)\|_{1} \frac{\mathcal{K}(\lambda)}{\lambda}.  
\end{equation}  

In view of \eqref{Adm-2}, this can alternatively be expressed as
\begin{equation}  
\label{eq:w_svf_form}  
w(\lambda) \underset{\lambda \to 0^+}{\sim} \|v(x,\cdot)\|_{1} \lambda^{-1-\varrho} L\left(\frac{1}{\lambda}\right),  
\end{equation}  
where $\varrho \geq 0$ and $L$ is a slowly varying function at infinity (as defined in \eqref{Adm-2}).

The positivity of $v^{E}$ implies that $\mathbf{F}$ is continuous and non-decreasing. Therefore, applying Theorem \ref{FKT} gives
\begin{equation}  
\label{eq:Tauberian_result}  
\mathbf{M}_t(v^E(x, t)) \underset{t \to \infty}{\sim} \frac{\|v(x,\cdot)\|_{1}}{\Gamma(\varrho +1)} t^{\varrho-1} L(t).  
\end{equation}  

Finally, combining this result with assumption \eqref{Adm-2} establishes \eqref{main-res}, which completes the proof of Theorem~\ref{Main-thm}.

\subsection{Proof of Corollary \ref{Spec-Kern}}
As noted in Remark~\ref{Classes}, all classes considered satisfy assumptions \eqref{Adm-1} and \eqref{Adm-2}, with the exception of class ($\mathbf{C}_3$). For ($\mathbf{C}_3$), the value $0$ in \eqref{Adm-1} is replaced by $\sqrt{ab}$. 

In view of this observation, the conclusion \eqref{Class-Examp} follows immediately from \eqref{main-res}, which completes the proof of Corollary~\ref{Spec-Kern}.
\subsection{Asymptotic Behavior for \texorpdfstring{$\alpha \in (0,1)$}{alpha in (0,1)} with \texorpdfstring{$\mathscr{L}_{1,0}$}{L1,0} Operators}
\label{1-0}
In this section, we will give the proof of Theorem \ref{alpha1-0}.
\begin{proof}[Proof of Theorem~\ref{alpha1-0}]
The solution to problem~\eqref{alpha-1-0} is given by
\begin{equation}  
    \label{Solu-1-0}  
    v(x,t) = \left[\mathbf{Z}_{\alpha}\left(\frac{t^{\gamma+1}}{\gamma+1}\right) \ast \varphi\right](x),  
\end{equation}  
where $\mathbf{Z}_{\alpha}$ is the fundamental solution of \eqref{alpha11} given by \eqref{Z-alpha}.  

Since $\varphi \geq 0$ and $\alpha \in (0,1)$, \eqref{decay-Z-alpha} ensures $v \geq 0$. Furthermore, as $\varphi \in L^1(\mathbb{R}^N)$, \eqref{decay-Z-alpha} yields the decay estimate  
\begin{equation}  
    \label{v-L1}  
    0 \leq v(x,t) \leq C \|\varphi\|_1 \, t^{-\frac{N\alpha(\gamma+1)}{2}}, \quad (x,t) \in \mathbb{R}^N\times (0,\infty).  
\end{equation}  
This implies $v(x,\cdot) \in L^1(1,\infty)$ provided $\gamma > \frac{2}{N\alpha} - 1$.  

Additionally, since $\varphi \in L^\infty(\mathbb{R}^N)$, the decay estimate ~\eqref{decay-Z-alpha} gives the uniform bound  
\begin{equation}  
    \label{v-Linfty}  
    0 \leq v(x,t) \leq \|\varphi\|_\infty, \quad (x,t) \in \mathbb{R}^N\times (0,\infty).  
\end{equation}  
Combining \eqref{v-L1} and \eqref{v-Linfty}, we deduce that $v$ is nonnegative and satisfies $v(x,\cdot) \in L^1(0,\infty)$ whenever $\gamma > \frac{2}{N\alpha} - 1$. The conclusion now follows directly from Theorem~\ref{Main-thm}.  
\end{proof}

\subsection{Asymptotic Behavior when \texorpdfstring{$\alpha =1$}{alpha =1} for \texorpdfstring{$\mathcal{L}_{1,1}$} {L1,1} Operators}

\label{1-1}
This section is devoted to the proof of Theorem \ref{11-1}.
\begin{proof}[Proof of Theorem~\ref{11-1}]
The solution to problem~\eqref{1-1-1} is expressed as
\begin{equation}
    \label{Sol-111}
    v(x,t) = \left[\mathbf{E}_{\ts}\left(\frac{t^{\gamma+1}}{\gamma+1}\right) \ast \varphi\right](x),
\end{equation}
where $\mathbf{E}_{\ts}$ denotes the fundamental solution of the mixed local/nonlocal diffusion equation \eqref{eq-1-1-1}.

From Lemma \ref{E-s}, since $\varphi \geq 0$ and $\mathbf{E}_{\ts} \geq 0$, it immediately follows that $v \geq 0$.

Furthermore, applying \eqref{eq:conv_est} yields the  decay estimate
\begin{equation}
    \label{v-L1-1}
    0 \leq v(x,t) \leq C\|\varphi\|_1 t^{-\frac{N(\gamma+1)}{2\ts}}, \quad (x,t) \in \mathbb{R}^N\times (0,\infty).
\end{equation}
This implies $v(x,\cdot) \in L^1(1,\infty)$ provided that $\gamma > \frac{2\ts}{N} - 1$.

For the uniform bound, since $\varphi \in L^\infty(\mathbb{R}^N)$, we obtain from  \eqref{eq:conv_est}  that
\begin{equation}
    \label{vv-Linfty}
    0 \leq v(x,t) \leq C\|\varphi\|_\infty, \quad (x,t) \in \mathbb{R}^N\times (0,\infty).
\end{equation}

Combining \eqref{v-L1-1} and \eqref{vv-Linfty}, we conclude that $v$ is nonnegative and satisfies $v(x,\cdot) \in L^1(0,\infty)$ when $\gamma > \frac{2\ts}{N} - 1$.
The theorem follows by applying Theorem~\ref{Main-thm}.
\end{proof}
\section{Concluding Remarks}
\label{CR}
The study of subordinated solutions to fractional diffusion equations with mixed local and nonlocal operators provides a useful framework for understanding complex systems that exhibit memory effects and anomalous diffusion. By applying techniques from probability theory, asymptotic analysis, and partial differential equations (PDEs), we have analyzed the long-term behavior of these solutions, particularly through the Cesàro mean, which gives a reliable measure of their average dynamics.

Our results both unify and expand upon classical fractional dynamics, illustrating how different types of subordinators, such as stable, gamma, and tempered stable processes, shape the asymptotic behavior of solutions. The explicit formulas that we derived for various kernel classes demonstrate the flexibility of this approach in modeling real-world phenomena, ranging from biological systems to materials science.

Future research could explore nonlinear extensions of these equations, incorporating subordination with more complex operators or boundary conditions. In addition, the enhanced modeling capabilities of this framework have significant potential for applications in ecology, epidemiology, and finance. The interplay between stochastic processes and deterministic PDEs, as demonstrated in our work, offers a promising direction for both theoretical advances and practical implementations.

\vspace{1cm}

\hrule 

\vspace{0.5cm}

\noindent{\bf\large Declarations.} {\em On behalf of all authors, the corresponding author states that there is no conflict of interest. No data-sets were generated or analyzed during the current study.}
\vspace{0.3cm}
 \hrule

\end{document}